%% file: subexp.tex
\documentclass[12pt]{amsart}

\input{header/ams_header}

\usepackage{braket}

\newcommand{\norm}[1]{\left\lVert#1\right\rVert}
\newcommand{\Free}{\mcF}
\DeclareMathOperator{\hlp}{hlp}
\DeclareMathOperator{\tr}{tr}
\DeclareMathOperator{\Cl}{Cliff}

\title{A group with at least subexponential hyperlinear profile}

\author{William Slofstra}
\thanks{Institute for Quantum Computing and Department of
    Pure Mathematics, University of Waterloo, Waterloo, Canada. email:
    \texttt{weslofst@uwaterloo.ca}}

\begin{document}

\begin{abstract}
    The hyperlinear profile of a group measures the growth rate of the
    dimension of unitary approximations to the group. We construct a
    finitely-presented group whose hyperlinear profile is at least
    subexponential, i.e. at least $\exp({1/\eps^{k}})$ for some $0 < k < 1$. 
    We use this group to give an example of a two-player non-local 
    game requiring subexponential Hilbert space dimension to play
    near-perfectly. 
\end{abstract}

\maketitle

\section{Introduction}

Let $G = \langle S : R \rangle$ be a finitely-presented group. An
\emph{approximate representation} of $G$ is a unitary representation of the
free group $\Free(S)$ in which the defining relations of $G$ are close to the
identity. To quantify this, let $\norm{\cdot}_f$ be the normalized Frobenius
norm, i.e. if $A$ is a $d \times d$ matrix then $\norm{A}_f := \sqrt{\tr(A^*
A)/d}$.  A \emph{$d$-dimensional $\eps$-representation of $G$} is a homomorphism
$\phi : \Free(S) \arr \mcU(\C^d)$ such that
\begin{equation*}
    \norm{\phi(r) - 1}_f \leq \eps
\end{equation*}
for all $r \in R$.  A word $w \in \Free(S)$ is said to be \emph{non-trivial in
approximate representations} if there is some $\delta > 0$ such that for all
$\eps > 0$, there is an $\eps$-representation $\phi$ with 
\begin{equation*}
    \norm{\phi(w) - \Id}_f \geq \delta.
\end{equation*}
In other words, $w$ is non-trivial in approximate representations if there is a
sequence of $\eps$-representations with $\eps \arr 0$, in which $w$ is bounded
away from the identity. The dimensions of the approximate representations in
this sequence do not have to be constant, and indeed, unless the element
represented by $w$ is non-trivial in some finite-dimensional exact
representation of $G$, the dimensions of the approximate representations in
this sequence must go to infinity. The hyperlinear profile of $G$ measures the
growth rate of dimensions in such a sequence:
\begin{defn}[\cite{SV17}]
    Given $X \subset \Free(S)$, define
    \begin{equation*}
        \hlp(X) : \R_{>0} \times \R_{>0} \arr \mbN \cup \{+\infty\}
    \end{equation*} 
    by setting $\hlp(X;\delta,\eps)$ to be the smallest integer $d$ for which
    there is a $d$-dimensional $\eps$-representation $\phi$ such that 
    $\norm{\phi(w) - \Id}_f \geq \delta$ for all $w \in X$ (or $+\infty$ if no
    such integer $d$ exists). 

    The \emph{hyperlinear profile of $G$} is the collection of functions
    $\hlp(X)$, where $X$ ranges across finite subsets of $\Free(S)$ not
    containing any words which are trivial in $G$. We say that the hyperlinear
    profile of $G$ is $\geq f(\delta,\eps)$ if there is some set $X$, not
    containing any words which are trivial in $G$, such that
    $\hlp(X;\delta,\eps) \geq f(\delta,\eps)$ for all $\delta,\eps > 0$. 
\end{defn}
A group is said to be \emph{hyperlinear} if every non-trivial element is
non-trivial in approximate representations. It is a major open question to
determine whether there is a non-hyperlinear group; such a group would provide
a counterexample to the Connes embedding problem. The difficulty of this
question can perhaps be seen in the number of other open questions of this type
in the literature. Indeed, for any class $\mcC$ of groups with bi-invariant
metrics, there is a class of $\mcC$-approximable groups defined analogously to
hyperlinear groups, with the metric groups in $\mcC$ replacing the unitary
groups with the norm $\norm{\cdot}_f$ \cite{Th12, DCGLT17}.\footnote{For general
classes $\mcC$, it is actually necessary to adjust this definition a bit to
properly handle approximations to finite sets; see \cite{Th12} for the precise
definition. For well-behaved classes, such as unitary groups with metric
induced by $\norm{\cdot}_f$ or the operator norm, the definition used above is
sufficient.} For instance, if $\mcC$ is the class of unitary groups with
metric induced from the operator norm, then $\mcC$-approximable groups are
known as \emph{MF groups}. When $\mcC$ is the class of symmetric groups with
the Hamming metric, then $\mcC$-approximable groups are known as \emph{sofic
groups}.  Just as with hyperlinear groups, it is not known if there is are
non-MF or non-sofic groups. Recently, De Chiffre, Glebsky, Lubotzky, and Thom
have shown that there is a non-$\mcC$-approximable group when $\mcC$ is the
class of unitary groups with the unnormalized Frobenius norm \cite{DCGLT17}.
Thom has also shown that there is a non-$\mcC$-approximable group when $\mcC$
is the class of finite groups with commutator-contractive invariant length
functions \cite{Th12}. These seem to be the only natural classes $\mcC$ of
metric groups where the existence of non-$\mcC$-approximable groups has been
settled. 

In terms of hyperlinear profile, a group $G = \langle S : R\rangle$ is
non-hyperlinear if and only if there is a word $w \in \mcF$ and a $\delta > 0$
that $\hlp(w;\delta, \eps)$ is infinite for all $\eps$ close to zero. Given the
apparent difficulty in finding a non-hyperlinear group, it seems interesting to
look for lower bounds on $\hlp(w;\delta,\eps)$ in specific groups and for
specific words $w$, especially lower bounds that show that
$\hlp(w;\delta,\eps)$ grows fast. Lower bounds on the order of
$\sqrt{\ln(1/\eps)}$, $1/\eps^{2/3}$, and $1/\eps^2$ have been given in
\cite{Fr13}, \cite{SV17}, and \cite{DCGLT17}, respectively. The point of this
note is to construct a group where we can show a subexponential lower bound on
hyperlinear profile:
\begin{thm}\label{T:main}
    There is a finitely-presented group $\mcG = \langle S : R \rangle$, a word $w
    \in \mcF(S)$, and constants $\alpha, C, C' > 0$, such that 
    \begin{equation*}
        \hlp(w; \delta,\eps) \geq C' \cdot \exp\left[C \cdot \left(\frac{\delta}{\eps}\right)^{\!\! \alpha} \right]
    \end{equation*}
    for all $\eps,\delta > 0$. 
\end{thm}
In particular, the hyperlinear profile of $\mcG$ is superpolynomial. In the
proof of Theorem \ref{T:main}, we can take $\alpha$ to be any constant $< 1/2$,
although the value of $C$ will depend on $\alpha$.  Note that we do not give an
upper bound on $\hlp(w;\delta,\eps)$ in Theorem \ref{T:main}. In fact, we do
not even know if the constructed group is hyperlinear (although we have no
reason to suspect that it is not).  

The notion of hyperlinear profile is inspired by the notion of sofic profile
introduced by Cornulier \cite{Co13}, which measures the dimension growth rate
of approximate representations in permutation groups, rather than unitary
groups. In Cornulier's definition, the sofic profile of a word is indexed by an
integer parameter $n$, corresponding to the parameters $\eps = \sqrt{2/n}$ and
$\delta = \sqrt{2-2/n}$ in our definition of hyperlinear profile. For words
which are not non-trivial in finite-dimensional representations, the sofic
profile is always at least linear in $n$. Cornulier has asked whether it is
possible for a group to have a superlinear sofic profile.  As shown in
\cite{SV17}, hyperlinear profile is a lower bound on
sofic profile.  As a result, the group in Theorem \ref{T:main} also has
subexponential sofic profile, answering Cornulier's question:
\begin{cor}\label{C:main}
    There is a finitely-presented group $\mcG$, a finite subset $E \subset
    \mcG$, and constants $\alpha, C,C' > 0$, such that the sofic profile
    $\sigma(E;n) \geq C' \cdot \exp\left(C n^{\alpha}\right)$ for all $n \geq 1$.  
\end{cor}
We can also ask about lower bounds in other norms, for instance, on the profile
with respect to the operator norm. The proof of Theorem \ref{T:main} is based
on Gowers and Hatami's stability theorem for approximate representations of
finite groups, which apply for any Schatten $p$-norm (\cite{GH17}, see also
\cite{DOT17}). Finite groups are also known to be stable in the operator norm
by a result of Kazhdan \cite{Ka82}. Consequently, Theorem \ref{T:main} holds
when the normalized Frobenius norm is replaced by the operator norm or another
normalized Schatten $p$-norm. However, for simplicity of presentation we focus
on the normalized Frobenius norm, except to note where the proof needs to be
changed for other norms. 

The bounds on hyperlinear profile in \cite{SV17} are used to show that there is
a non-local game with optimal quantum value $1$, but for which any
$\eps$-optimal strategy requires local Hilbert spaces of dimension
$\Theta(\eps^{1/k})$ (where $k$ is some fixed integer). Every
finitely-presented group embeds in a solution group by \cite[Theorem
3.1]{Sl16}, so Theorem \ref{T:main} can also be used to prove lower bounds on
local Hilbert space dimension in a non-local game: 
\begin{cor}\label{C:game}
    There are constants $C, C', \alpha > 0$ and a two-player non-local game $G$
    with commuting-operator value $1$, such that playing $G$ with success
    probability $1-\eps$ requires the players to use local Hilbert spaces of
    dimension at least $C' \cdot \exp\left(C / \eps^{\alpha} \right)$.
\end{cor}
Note that Corollary \ref{C:game} has a caveat: while the game constructed via
this method will have commuting-operator value $1$, it is not clear if it will
have quantum value $1$. This is in part because it is not clear if the group in
Theorem \ref{T:main} is hyperlinear, and in part because of the limitations of
\cite[Theorem 3.1]{Sl16}. Ji, Leung, and Vidick have recently given a simpler
three-player non-local game with similar subexponential entanglement
requirements, but without this caveat \cite{JLV18}. 

The rest of the paper is split into two parts. In the first part (Section
\ref{S:construction}), we construct the group $\mcG$. In the second part
(Section \ref{S:proof}) we prove Theorem \ref{T:main} and Corollaries
\ref{C:main} and \ref{C:game}.

\subsection{Acknowledgements}

I thank Tobias Fritz, Andreas Thom, and Thomas Vidick for helpful comments. 

\section{The construction}\label{S:construction}

Before proving Theorem \ref{T:main}, we explain how to construct the group
$\mcG$.  To start, let $\Cl(n)$ denote the Clifford algebra of rank $n$, i.e.
\begin{align*}
    \Cl(n) = \C \langle x_1,\ldots,x_n : x_i^2 = 1, x_i x_j = - x_j x_i \text{ for all }
        1 \leq i \neq j \leq n \rangle.
\end{align*}
To think about $\Cl(n)$ in the language of groups, we can represent $-1$ by a 
central involution $J$, and look at the group
\begin{align*}
    \mcC(n) = \langle J, x_1,\ldots,x_n\; :\ &J^2 = 1, x_i^2 = [x_i,J] = 1 \text{ for all }
        1 \leq i \leq n, \\
        & [x_i,x_j] = J \text{ for all } 1 \leq i \neq j \leq n \rangle.
\end{align*}
The algebra $\Cl(n)$ is then the quotient of the group algebra of $\mcC(n)$ by
the relation $J=-1$. It is well-known that $\Cl(n)$ has either one or two
irreducible representations of dimension $2^{\lfloor n/2 \rfloor}$, and as we
will see in the next section, it is possible to lower bound the dimension of
approximate representations as well. To prove Theorem \ref{T:main}, we need
to combine the groups $\mcC(n)$ into one algebraic object. This can be done
by considering the group
\begin{align*}
    \mcC = \langle J, x_i, i \in \Z \; :\ & J^2 = 1, x^2_i = [x_i,J] = 1 \text{ for all } i \in \Z, \\
        & [x_i, x_j] = J \text{ for all } i, j \in \Z \text{ with } i\neq j\rangle
\end{align*}
corresponding to the infinite-rank Clifford algebra. 
Note that the word problem in $\mcC$ is easily solvable, as every element has a
normal form $J^a x_{i_1} \cdots x_{i_k}$ for $a \in \Z_2$ and $i_1 < \ldots <
i_k$. Another consequence of this normal form is that the subgroup $\langle
x_1, \ldots, x_n \rangle$ in $\mcC$ can be identified with $\mcC(n)$. 

Indexing the generators of $\mcC$ by $\Z$ is convenient for the following
reason: any bijection $\sigma : \Z \arr \Z$ gives an automorphism of $\mcC$
sending $x_i \mapsto x_{\sigma(i)}$ and $J \mapsto J$. Let $\sigma(n) = n+1$,
and let $\mcK_0$ be the semidirect product $\Z \ltimes_{\sigma} \mcC$. If $z$
is the generator of the $\Z$ factor, then $x_i = z^i x_0 z^{-i}$, so $\mcK_0$
is finitely generated by $z$ and $x_0$ (although we will always include $J$ in
the list of generators for convenience). 
Finally, we let 
\begin{equation*}
    \mcK = \langle \mcK_0, t : t z t^{-1} = z^2 \rangle
\end{equation*}
be the HNN extension of $\mcK_0$ by the endomorphism of $\langle z \rangle \iso
\Z$ sending $z \arr z^2$. We regard $\mcK$ as a finitely-generated group with
generator set $\{J,t,x_0,z\}$. 

We want to embed $\mcK$ in a finitely presented group $\mcG =
\langle S : R \rangle$.  Since $\mcK$ is recursively presented, this is
possible by Higman's embedding theorem \cite{Hi61}. However, to work with
approximate representations, we want an embedding where the defining relations
of $\mcK$ can be inferred from the defining relations $R$ of
$\mcG$ in an efficient manner. Thus we embed $\mcK$ in a finitely-presented
group using a quantitative version of Higman's theorem due to Birget,
Ol'shanskii, Rips, and Sapir \cite{BORS02}.  To state this theorem, we recall
the notion of an isoperimetric function. Let $G = \langle S : R \rangle$ be a
finitely-presented group. In elementary terms, a word $w \in \mcF(S)$ is
trivial in $G$ if and only if there is a sequence $w = w_0 \arr w_1 \arr \cdots
\arr w_k = 1$ of words in $\mcF(S)$ such that $w_{i}$ differs from $w_{i-1}$ by
the application of a relation in $R$ for all $i=1,\ldots,k$, i.e. we can get
$w_{i}$ from $w_{i-1}$ by replacing a subword $b$ of $w_{i-1}$ with $a^{-1}
c^{-1}$, where $abc = 1$ is some relation in $R$, and then reducing the
resulting word in $\mcF(S)$.  The \emph{area of $w$} is the smallest $k$ for
which such a sequence exists. A function $f : \mbN \arr \mbN$ is an
\emph{isoperimetric function for $G$} if the area of $w$ is $\leq f(n)$ for all
words $w \in \mcF(S)$ of length $\leq n$ which are trivial in $G$. 

\begin{thm}[\cite{BORS02}]\label{T:BORS}
    Let $K$ be a finitely-generated group with word problem solvable in
    (non-deterministic) time $\leq T(n)$, where $T(n)^4$ is at least
    superadditive, i.e. $T(m+n)^4 \geq T(m)^4 + T(n)^4$. Then $K$ embeds
    into a finitely-presented group $G$ with isoperimetric function 
    $n^2 T(n^2)^4$.
\end{thm}
A Turing machine $M$ solves the word problem for a group $K$ with respect to a
finite generating set $S$ if the input alphabet of $M$ is $\{s, s^{-1} : s \in
S \}$ and $M$ accepts a word $w$ if and only if $w$ is trivial in $K$. Saying
that the word problem for $K$ is solvable in time $\leq T(n)$ means that there
is a Turing machine $M$ solving the word problem for $K$, such that for all $n
\geq 1$, if $w$ is a word of length $\leq n$ which is trivial in $K$, then $M$
accepts $w$ in at most $T(n)$ steps.\footnote{In other words, the
non-deterministic time complexity of $M$ is $\leq T(n)$. For deterministic
machines, this is a priori weaker than saying that the deterministic time
complexity is $\leq T(n)$, since we don't require $M$ to halt in time $\leq T(n)$
for words which are non-trivial in $K$. However, the distinction will be
irrelevant in our application of Theorem \ref{T:BORS}.}
\begin{lemma}\label{L:wordproblem}
    The word problem for $\mcK$ with respect to the generating set $\{J,t,x_0,z\}$
    can solved in $O(n^3)$ time (on a multi-tape deterministic Turing machine). 
\end{lemma}
Lemma \ref{L:wordproblem} is fairly intuitive; however, there are some
subtleties to representing group elements on a Turing machine, so we give a
detailed proof.
\begin{proof}[Proof of Lemma \ref{L:wordproblem}]
    As previously mentioned, any element of $\mcC$ can be written uniquely as
    $J^b x_{i_1} \cdots x_{i_{\ell}}$, where $b \in \Z_2$ and $i_1 < \ldots <
    i_{\ell}$.  As a
    semidirect product, any element of $\mcK_0$ thus has a unique normal form
    $z^a J^b x_{i_1} \cdots x_{i_{\ell}}$, where $a \in \Z$, $b \in \Z_2$, and
    $i_1 < \ldots < i_{\ell}$. On a Turing machine, a signed integer $a$ can be
    represented as a string of bits, with the first bit representing the sign.
    The bit length of $a$ is the number of bits needed to represent $a$, i.e.
    $\lceil \log_2 |a|\rceil + 2$ (assuming $a \neq 0$).  The normal form $z^a
    J^b x_{i_1} \cdots x_{i_{\ell}}$ can be represented by the list of numbers
    $a,b,i_1,\ldots, i_{\ell}$, with the different values separated by empty
    cells. We say that $z^a J^b x_{i_1} \cdots x_{i_{\ell}}$ has list-length 
    $\ell$, and maximum bit length given by the maximum bit length of 
    $a,i_1,\ldots,i_{\ell}$ (the bit length of $b$ is just $1$). Thus a
    normal form with list-length $\ell$ and maximum bit length $B$ can be
    represented in at most $(\ell+1)B+2$ cells.  Note that this is different
    from the representation of input words in the word problem, where, for
    instance, the number of cells used to represent $z^a$ is $a$ rather than
    $\lceil \log_2 |a| \rceil +4$. 

    Suppose we have two normal forms $w_1 = z^{a_1} J^{b_1} x_{i_1} \cdots x_{i_{\ell_1}}$
    and $w_2 = z^{a_2} J^{b_2} x_{j_1} \cdots x_{j_{\ell_2}}$
    for elements of $\mcK_0$, represented as described above, and want to find the normal form of $w_1 \cdot
    w_2$. Assume first that $a_1 = a_2=0$ and $b_1 = b_2 = 0$, so that we just
    need to find the normal form $J^b x_{k_1} \cdots x_{k_{\ell_3}}$ of
    $x_{i_1} \cdots x_{i_{\ell_1}} \cdot x_{j_1} \cdots x_{j_{\ell_2}}$.  Since
    the lists $i_1,\ldots,i_{\ell_1}$ and $j_1,\ldots,j_{\ell_2}$ are ordered,
    the list $k_1,\ldots,k_{\ell_3}$ can be found by \emph{merging} the two input lists.
    Merging means that we perform an iterative algorithm, where in each step we
    compare the first two elements of each input list, removing the smaller of
    the two from its list and adding it to the end of an output list. If the
    two elements are equal, we remove both and add nothing to the output. The
    process then continues with the remaining input list elements until no
    elements remain in either list, at which point the output reads
    $k_1,\ldots,k_{\ell_3}$. We can keep track of the power $b$ of $J$ during this
    process by also tracking the parity $p \in \Z_2$ of the number of remaining
    elements in the first input list. So at the beginning, $b=0$ and $p =
    \ell_1$, whenever an element is removed from the first input list we add
    $1$ to $p$, and whenever an element is removed from the second input list
    we add $p$ to $b$.  Comparison of two integers of bit length $B$ can be
    done in time $O(B)$, so if maximum bit length of $w_1$ and $w_2$ is $\leq
    B$, then this entire process takes time $O((\ell_1 + \ell_2)B)$. For general
    $a_1,a_2,b_1,b_2$,
    \begin{equation*}
        w_1 \cdot w_2 = z^{a_1 + a_2} J^{b_1 + b_2} x_{i_1 - a_2} \cdots
            x_{i_{\ell_1}-a_2} \cdot x_{j_1} \cdots x_{j_{\ell_2}}.
    \end{equation*}
    Addition of signed integers of bit length $B$ also takes time $O(B)$, and
    the sum will have bit length at most $B+1$. So the normal form of $w_1
    \cdot w_2$ can be found in time $O((\ell_1+\ell_2)B)$, and will have
    list length $\leq \ell_1 + \ell_2$ and maximum bit length $\leq B+1$. 
    
    Next we look at the problem of finding the normal form of an element of
    $\mcK_0$ given as a word $w$ over $\{J^{\pm 1}, x_0^{\pm 1}, z^{\pm 1}\}$.
    Suppose $w$ has length $n$. By collecting powers of $J$ and $z$, we can
    find a $b \in \Z_2$ and a list of signed integers $a_0,\ldots,a_m$ such
    that $w = J^b z^{a_0} x_0 z^{a_1} x_0 \cdots x_0 z^{a_m}$ in $\mcK_0$. This
    takes time $O(n \cdot B)$, where $B$ is the maximum bit length of the
    integers $a_0,\ldots,a_m$. Then $w = z^{k_0} J^b x_{-k_1} \cdots x_{-k_m}$
    in $\mcK_0$, where $k_j = \sum_{i=j}^{m} a_i$ for each $j=0,\ldots,m$. The
    list $k_0,k_1,\ldots,k_m$ can be found in time $O(n \cdot B')$, where
    $B'$ is the maximum bit length of the integers $a_i,k_i$, $i=0,\ldots,m$.
    The normal form of $w$ can then be found by sorting the list $-k_1,
    \ldots,-k_m$ and keeping track of the power of $J$. If we use merge sort
    with the merge procedure mentioned above, then this takes time
    $O(m\log(m)\cdot B')$. Since $m$ and all the integers $a_i,k_i$,
    $i=0,\ldots,m$ are bounded by $n$, we conclude that the normal form of $w$
    can be found in time $O(n (\log n)^2)$, and will have list length $\leq n$
    and maximum bit length $O(\log(n))$. 

    Finally, suppose we are given a word $w$ in $\{J^{\pm 1}, t^{\pm 1},
    x_0^{\pm 1}, z^{\pm 1}\}$ of total length $n$, in which $t^{\pm 1}$ occurs
    $k$ times. We can write $w$ as $w_0 t^{c_1} w_1 t^{c_2} \cdots t^{c_k}
    w_k$, where $c_1,\ldots,c_k \in \{\pm 1\}$, and $w_0,\ldots,w_k$ are words
    over $\{J^{\pm 1}, x_0^{\pm 1}, z^{\pm 1}\}$. Using the method of the last
    paragraph, we can find the normal form $w_i'$ of $w_i$ for each
    $i=0,\ldots,k$, in total time $O(n (\log n)^2)$. This gives us a sequence
    of integers $c_1,\ldots,c_k \in \{\pm 1\}$ and normal forms $w_0',\ldots,w_k'$
    such that $k \leq n$, $w = w_0' t^{c_1} w_1' \cdots t^{c_k} w_k'$ in
    $\mcK$,the sum of the list lengths of $w_0',\ldots,w_k'$ is $\leq n$, and
    the maximum bit lengths of $w_0',\ldots,w_k'$ are $\leq B = O(\log(n))$. By
    Britton's lemma, if $w$ is trivial in $\mcK$, then either $k=0$ and the normal form
    $w_0'$ is trivial, or there is an index $i \in \{1,\ldots,k-1\}$ such that either
    (1) $c_{i} = 1$, $c_{i+1} = -1$, and $w_i' \in \langle z \rangle$, or (2)
    $c_{i} = -1$, $c_{i+1} = 1$, and $w_i' \in \langle z^2 \rangle$. The only
    normal forms which belong to $\langle z \rangle$ (resp. $\langle z^{2}
    \rangle$) are those of the form $z^a$ (resp. $z^{2a}$) for some $a \in \Z$,
    so if an index $i$ of type (1) or (2) exists, we can find it in $O(n B)$ time.
    Suppose we are given an index $i$ of type (1), and let $w_i = z^a$. Then $t
    z^a t^{-1}$ is equal to $z^{2a}$ in $\mcK$, so if $w''$ is the normal form of
    $w_{i-1} z^{2a} w_{i+1}$, then the word $w_0' t^{c_1} w_1' \cdots w_{i-2}'
    t^{c_{i-1}} w'' t^{c_{i+2}} w_{i+2}' \cdots w_{k}'$ is equivalent to $w$ in $\mcK$.
    In this new word, $t$ and $t^{-1}$ occur $k-2$ times, and the sum of the
    list lengths of $w_0',\ldots,w'',\ldots,w_k'$ is no more than the sum of
    the list lengths of $w_0',\ldots,w_k'$, so both quantities are bounded by
    $n$. However, the maximum bit length of $z^{2a}$ is one more than the
    maximum bit length of $w_i$, and the bit length can go up again when
    calculating the normal form of $w_{i-1} z^{2a} w_{i+1}$. But, the normal
    form $w''$ can still be found in $O(n B)$ time, and the maximum bit lengths
    of $w_0',\ldots,w'',\ldots,w_k'$ will be $\leq B+3$. Given an index of type
    (2), we can do the same thing, but replacing $w_i = z^{2a}$ with $z^a$. If
    we repeat this process, we will conclude after at most $k/2 = O(n)$ steps
    that either $w$ is trivial in $\mcK$, or $w$ is non-trivial by Britton's
    lemma. Since the maximum bit lengths $B$ increase by at most three in each
    step, we have $B = O(n)$ throughout this process, so each step takes at
    worst $O(n^2)$ time, and we can determine whether $w$ is trivial in
    $O(n^3)$ time.
\end{proof}
\begin{rmk}
    While this proof includes some overestimates, the maximum bit length can
    indeed reach order $O(n)$, since $t^k z t^{-k} = z^{2^k}$ in $\mcK$. 
    As we will see, this is the point of including $t$ in the construction
    of $\mcK$. For instance, including $t$ in $\mcK$ allows us to represent
    $x_{2^k}$ by $t^k z t^{-k} x_0 t^k z^{-1} t^{-k}$, a word of length $O(k)$,
    rather than by $z^{2^k} x_0 z^{-2^k}$, a word of length $O(2^k)$. 
\end{rmk}
Applying Theorem \ref{T:BORS} to $\mcK$, we get the group $\mcG$ used in
Theorem \ref{T:main}.
\begin{cor}\label{C:construction}
    There is an embedding of $\mcK$ in a finitely-presented group $\mcG$ with
    isoperimetric function $C n^{26}$ for some constant $C$.
\end{cor}

\section{Proof of Theorem \ref{T:main}}\label{S:proof}

To prove Theorem \ref{T:main}, we use the fact that any approximate
representation of $\Cl(n)$ must have dimension near $2^{\lfloor n/2 \rfloor}$.
To state this fact in the form that we will use, we need some additional
terminology. An \emph{$\eps$-homomorphism from a group $G$ to $\mcU(\C^n)$} is
a function $\psi : G \arr \mcU(\C^d)$ such that $\norm{\psi(gh) -
\psi(g)\psi(h)}_f \leq \eps$ for all $g,h \in G$. Note that for infinite
groups, this is a much stronger notion of approximate representation than the
notion introduced in the introduction.
\begin{lemma}\label{L:clifford}
    If $\psi$ is an $\eps$-homomorphism from $\mcC(n)$ to $\mcU(d)$ such that
    \begin{equation*}
        \norm{\psi(J) - \Id}_f > 42\eps, 
    \end{equation*}
    then $d \geq 2^{\lfloor n/2 \rfloor - 1}$. 
\end{lemma}
Other variants of Lemma \ref{L:clifford} have been used previously in quantum
information (for instance, see \cite{ostrev2016entanglement}). This version of the result follows
easily from the Gowers-Hatami stability theorem.  We give a complete proof for
the convenience of the reader. 
\begin{theorem}[\cite{GH17}, see \cite{DOT17} for this exact formulation]\label{T:gowershatami}
    Let $G$ be a finite group. If $\psi$ is an $\eps$-homomorphism from $G$
    to $\mcU(\C^d)$, where $\eps < 1/16$, then there is an exact representation
    $\gamma : G \arr \mcU(\C^m)$ for some $d \leq m \leq (1-4\eps^2)^{-1} d$, as
    well as an isometry $U : \C^d \arr \C^m$, such that 
    \begin{equation*}
        \norm{\psi(g) - U^* \gamma(g) U}_f \leq 42 \eps
    \end{equation*}
    for all $g \in G$. 
\end{theorem}
\begin{proof}[Proof of Lemma \ref{L:clifford}]
    Suppose that $\psi$ is an $\eps$-homomorphism from $\mcC(n)$ to $\mcU(d)$
    with $\norm{\psi(J) - \Id}_f > 42\eps$. Since the distance between any two
    unitaries in the $\norm{\cdot}_f$ distance is at most $2$, we have that
    $\eps < 1/21 < 1/16$. Since $\mcC(n)$ is finite, there is an exact
    representation $\gamma : G \arr \mcU(\C^m)$ and isometry $U : \C^d \arr
    \C^m$ as in Theorem \ref{T:gowershatami}. To simplify formulas, note that
    $(1-4\eps^2) > 1/2$, so $m \leq 2 d$. 

    Now $U U^*$ is a projection, so 
    \begin{equation*}
        \norm{U^* A U}_f = \sqrt{\frac{1}{d} \tr(U^* A^* U U^* A U)}
            \leq \sqrt{\frac{1}{d} \tr(A^* A)} = \sqrt{\frac{m}{d}} \norm{A}_f
    \end{equation*}
    for any $m \times m$ matrix $A$. Hence 
    \begin{align*}
        \norm{\gamma(J)-\Id}_f 
            & \geq \sqrt{\frac{d}{m}} \norm{U^* (\gamma(J) -\Id) U}_f 
            = \sqrt{\frac{d}{m}} \norm{U^* \gamma(J) U - \Id}_f \\
            & \geq \sqrt{\frac{d}{m}} \left( \norm{\psi(J) - \Id}_f
                - \norm{\psi(J) - U^* \gamma(J) U}_f \right) \\
            & \geq \frac{1}{\sqrt{2}} \left( \norm{\psi(J) - \Id}_f - 42\eps\right)
                > 0.
    \end{align*}
    We can conclude that $\gamma(J) \neq \Id$. Since $\gamma(J)^2 = \Id$, the
    eigenvalues of $\gamma(J)$ are $\pm 1$, so the $(-1)$-eigenspace of
    $\gamma(J)$ must have positive dimension. Let $P$ be the projection onto
    this eigenspace. Since $J$ is central in $\mcC(n)$, $P \gamma P$ is a
    representation of $\mcC(n)$ (and hence of $\Cl(n)$) of dimension $\leq m$.
    Since all irreducible representations of $\Cl(n)$ have dimension
    $2^{\lfloor n/2 \rfloor}$, we conclude that $d \geq m/2 \geq 2^{\lfloor n/2
    \rfloor} / 2$. 
\end{proof}
\begin{rmk}\label{R:othernorms}
    Lemma \ref{L:clifford} also holds for the operator norm and the normalized
    Schatten $p$-norms (i.e.~the norms on $d \times d$ matrices of the form
    $\norm{A} = \left(\tr\left((A^* A)^{p/2}\right)/d\right)^{1/p}$ for some $p
    \geq 1$). Specifically, it follows from the results of \cite{Ka82} (for the
    operator norm) and \cite{GH17,DOT17} (for the $p$-norms) that if
    $\norm{\cdot}$ is one of these norms, then there are constants $C_0, C_1, k
    > 0$, such that if $\psi$ is an $\eps$-homomorphism from $\mcC(n)$ 
    to $\mcU(\C^d)$ with respect to $\norm{\cdot}$ such that $\eps \leq C_0$
    and $\norm{\psi(J) - \Id} > C_1 \eps$, then $d \geq 2^{\lfloor n/2
    \rfloor -k}$.
\end{rmk}

\begin{lemma}\label{L:uniformrep}
    If $\phi$ is a $d$-dimensional $\eps$-representation of $\mcC(n)$, then
    the function $\psi : \mcC(n) \arr \mcU(\C^d)$ defined by $\psi(g) = 
    \phi(\widetilde{g})$, where $\widetilde{g}$ is the normal form of $g$, is
    an $(n+1)^2 \eps$-homomorphism from $\mcC(n)$ to $\mcU(\C^d)$. 
\end{lemma}
\begin{proof}
    Suppose $\widetilde{g} = J^a x_{i_1} \cdots x_{i_k} x_j^b x_{i_{k+1}} \cdots
    x_{i_{m}}$ is a word in normal form, so $i_1 < \ldots < i_k < j < i_{k+1} <
    \ldots < i_m$ and $a,b \in \Z_2$. In $\mcC(n)$, we have that
    \begin{align*}
        x_j \cdot \widetilde{g} & = J^a x_j x_{i_1} \cdots x_j^b \cdots x_{i_m} 
            \quad (\text{at most $1$ application of } [x_j, J]=1) \\
        &= J^a J x_{i_1} J x_{i_2} \cdots J x_{i_k} x_j x_j^b \cdots x_{i_m}
            \quad (\text{$k$ applications of }x_j x_i = J x_i x_j) \\
        &= J^{a+k} x_{i_1} x_{i_2} \cdots x_{i_k} x_j x_j^b \cdots x_{i_m} 
            \quad (\text{$k$ applications of }[x_i,J]=1 \text{ and } J^2=1) \\
        &= J^{a+k} x_{i_1} \cdots x_{i_k} x_j^{b+1} \cdots x_{i_m} 
            \quad (\text{at most $1$ application of } x_j^2 = 1).
    \end{align*}
    Note that all exponents in the above equation are interpreted modulo $2$,
    and that to get the third line with only $k$ applications of the relations,
    we start by moving the last $J$ to the left, cancelling $J$'s as we go.
    Since $k \leq j-1$, we can put $x_j \cdot \widetilde{g}$ in normal form
    after at most $2(j-1) + 2 = 2j$ applications of the relations. The product
    $J \cdot \widetilde{g}$ can be put in normal form after applying the relation
    $J^2=1$ at most once. It follows that we can put the product $\widetilde{g_1}
    \cdot \widetilde{g_2}$ of two normal forms $\widetilde{g_1}$, $\widetilde{g_2}$ in
    normal form after at most $n(n+1) + 1 \leq (n+1)^2$ applications of the
    defining relations for $\mcC(n)$. If $\norm{\phi(r) - \Id}_f \leq \eps$
    for every defining relation $r$ of $\mcC(n)$, then 
    \begin{equation*}
        \norm{\phi(\widetilde{g_1 \cdot g_2}) -
            \phi(\widetilde{g_1})\phi(\widetilde{g_2})}_f \leq (n+1)^2 \eps
    \end{equation*}
    for all elements $g_1,g_2 \in \mcC(n)$. 
\end{proof}

\begin{proof}[Proof of Theorem \ref{T:main}]
    Let $\mcK$ be the group from the last section. By Corollary
    \ref{C:construction}, there is an embedding $\gamma$ of $\mcK$ in a
    finitely-presented group $\mcG = \langle S : R \rangle$ with isoperimetric
    function $C_0 n^{26}$ for some constant $C_0$. Choose a lift
    $\widetilde{\gamma} : \mcF(\{J,t,x_0,z\}) \arr \mcF(S)$ of $\gamma$, and
    let $C_1$ be the maximum of the lengths of the words
    $\widetilde{\gamma}(J)$, $\widetilde{\gamma}(t)$
    $\widetilde{\gamma}(x_0)$, and $\widetilde{\gamma}(z)$. If $w$ is a word
    over $\{J,t,x_0,z\}$ of length $m$, then $\widetilde{\gamma}(w)$ will have
    length $\leq C_1 m$. As a result, if $w$ is trivial in $\mcK$, then
    $\widetilde{\gamma}(w)$ will have area $\leq C_2 m^{26}$, where $C_2 = C_0
    C_1^{26}$. This means that $\widetilde{\gamma}(w)$ can be transformed to
    $1$ in $\mcF(S)$ by at most $C_2
    m^{26}$ applications of the defining relations $R$ of $\mcG$. One subtlety
    of this definition is that such a transformation may involve substitutions
    of the form $s s^{-1} \arr 1$, $s^{-1} s \arr 1$, $1 \arr s s^{-1}$, or $1
    \arr s^{-1} s$, $s \in S$, which are not counted towards the $C_2
    m^{26}$ bound. However, if $\phi$ is an $\eps$-representation of $\mcG$,
    then $\phi(s^{-1}) = \phi(s)^{-1}$ for all $s \in S$
    by definition.  Consequently, if $w$ is a (possibly non-reduced) word over
    $\{J,t,x_0,z\}$ of length $\leq m$ which is trivial in $\mcK$, then 
    \begin{equation}\label{E:isoperimetric}
        \norm{\phi(\widetilde{\gamma}(w)) - \Id}_f \leq C_2 m^{26} \eps.
    \end{equation}

    Given a positive integer $k$, set $z(k) := \prod_{i=0}^{\ell} t^i
    z^{k_i} t^{-i}$, where $k= \sum_{i=0}^{\ell} k_i 2^i$ is the binary 
    expansion of $k$, so $k_0,\ldots,k_{\ell} \in \{0,1\}$ and $k_{\ell}=1$. 
    Let $x(k) := z(k) x_0 z(k)^{-1}$, regarded as a word over
    $\{t^{\pm 1}, x_0, z^{\pm 1}\}$.  Since $\ell = O(\log_2 k)$, the length of
    $z(k)$ is $O((\log_2 k)^2)$, and hence the length of $x(k)$ is
    also $O((\log_2 k)^2)$. Of course, in $\mcK$ we have $z(k) = z^k$, and
    $x(k) = z^k x_0 z^{-k} = x_k$, a generator of $\mcC$. If $\phi$ is
    a $d$-dimensional $\eps$-representation of $\mcG$, then for every $n \geq
    2$, we can define an approximate representation $\psi_n :
    \Free(\{J,x_1,\ldots,x_n\}) \arr \mcU(\C^d)$ of $\mcC(n)$ by $\psi_n(J) =
    \phi(\widetilde{\gamma}(J))$ and $\psi_n(x_i) =
    \phi(\widetilde{\gamma}(x(i)))$. The words $x(i)^2$, $[x(i),J]$, and
    $J^{-1} [x(i), x(j)]$, $i \neq j$, are trivial in $\mcK$, and if we
    restrict to $1 \leq i \neq j \leq n$, have length $O((\log_2
    n)^2)$.  By Equation \eqref{E:isoperimetric}, there is a constant $C_3$ (independent
    of $\phi$)
    such that $\psi_n$ is a
    $C_3 (\log_2 n)^{52} \eps$-representation for all $n \geq 2$ (starting at $n=2$ avoids
    the technical problem that $\log_2 1 = 0$). By Lemma \ref{L:uniformrep},
    $\psi_n$ induces a $C_3 (n+1)^2 (\log_2 n)^{52}\eps$-homomorphism
    $\widetilde{\psi}_n$ from $\mcC(n) \arr \mcU(\C^d)$. Thus, for any $\kappa >
    0$ we can find a constant $C_4$ (again, independent of $\phi$) such that
    $\widetilde{\psi}_n$ is a $C_4 n^{2+\kappa}\eps$-homomorphism from $\mcC(n)
    \arr \mcU(\C^d)$ for all $n \geq 2$. 
    
    Finally, suppose $\phi$ is a $d$-dimensional $\eps$-representation of
    $\mcG$ with 
    \begin{equation*}
        \norm{\phi(\widetilde{\gamma}(J)) - \Id}_f \geq \delta > 0.
    \end{equation*}
    Then the approximate homormophisms $\widetilde{\psi}_n$ defined in the last
    paragraph have $\widetilde{\psi}_n(J) = \psi_n(J) = \phi(\widetilde{\gamma}(J))$,
    so $\norm{\widetilde{\psi}_n(J) - \Id}_f \geq \delta$.
    If we set, for instance,
    \begin{equation*}
        n = \left\lfloor \left( \frac{\delta}{84 C_4 \eps} \right)^{1/(2+\kappa)} \right\rfloor,
    \end{equation*}
    then $42 C_4 n^{2+\kappa} \eps < \delta$, so $d \geq 2^{\lfloor n/2 \rfloor -1}$
    by Lemma \ref{L:clifford}. Using the fact that $\lfloor x \rfloor \geq
    x-1$, we conclude that the theorem is true with $w =
    \widetilde{\gamma}(J)$, $C' = 2^{-5/2}$, $C = \ln(2)/2(84 C_4)^{1/(2+\kappa)}$, and
    $\alpha = 1/(2 + \kappa)$. 
\end{proof}
By Remark \ref{R:othernorms}, the proof also works for the operator norm and
the normalized Schatten $p$-norms, with different constants for $C'$ and $C$.
The exponent $\alpha$ is the same for any norm, and can be taken to be
$\alpha = 1/(2+\kappa)$ for any $\kappa > 0$.  We do not know if it is possible
to increase $\alpha$, say to $\alpha = 1/2$. On the other hand, if the value of
the exponent $\alpha$ is not important, then the proof of Theorem \ref{T:main}
can be shortened. First, we can simplify the construction of $\mcG$ by using
$\mcK_0$ in place of $\mcK$.  This simplifies Lemma \ref{L:wordproblem}, but
increases the $(\log_2 n)^{52}$ factor to a polynomial factor $n^p$. Second,
rather than using Lemma \ref{L:uniformrep} to turn approximate representations
into approximate homomorphisms, we can define the approximate homomorphisms
$\widetilde{\psi}_n$ in the proof directly as $\widetilde{\psi}_n(g) =
\phi(\widetilde{\gamma}(\widetilde{g}))$, where $\widetilde{g}$ is the normal
form $J^b x_{i_1} \cdots x_{i_k}$ of $g \in \mcC(n)$, and apply Equation
\eqref{E:isoperimetric} to the relations $\widetilde{g_1} \cdot \widetilde{g_2}
= \widetilde{g_1 \cdot g_2}$, in the full multiplication table of $\mcC(n)$,
rather than just to the defining relations. 

We finish with the proofs of Corollaries \ref{C:main} and \ref{C:game}.
\begin{proof}[Proof of Corollary \ref{C:main}]
    Let $\mcG = \langle S : R \rangle$ be the group from Theorem \ref{T:main},
    so that there is $w \in \mcF(S)$ and constants $\alpha, C_0, C_1$ such that
    \begin{equation*}
        \hlp(w;\delta,\eps) \geq C_1 \cdot \exp (C_0 (\delta / \eps)^{\alpha}). 
    \end{equation*}
    Fix $\delta < \sqrt{2}$. By \cite[Proposition 7.2]{SV17}, there is a finite
    subset $E \subseteq \mcG$ and constants $C_2,N > 0$ such that
    \begin{equation*}
        \sigma(E;n) \geq \hlp(w;\delta,C_2/\sqrt{n})
    \end{equation*}
    for all $n \geq N$. Since $\sigma(E;n) \geq 1$ for all parameters $E,n$, we
    can find constants $C',C > 0$ such that 
    \begin{equation*}
        \sigma(E;n) \geq C' \cdot \exp(C n^{\alpha/2} )
    \end{equation*}
    for all $n \geq 1$.
\end{proof}

\begin{proof}[Proof of Corollary \ref{C:game}]
    Let $\mcG$ be the group from Theorem \ref{T:main}, and take a presentation
    of $\mcG$ where the element $J$ is one of the generators. Let 
    \begin{equation*}
        \mcG' := \langle \mcG, s, J' : [\mcG,J'] = [s,J'] = (J')^2 = 1, [s,J]=J' \rangle.
    \end{equation*} 
    Note that  $\mcG'$ is the HNN extension of the subgroup $\langle \mcG, J'
    \rangle \iso \mcG \times \Z_2$ by the automorphism of $\langle J,J' \rangle
    \iso \Z_2 \times \Z_2$ sending $J' \mapsto J'$ and $J \mapsto J J'$. Hence
    $J'$ is a central element of order $2$ in $\mcG'$. Also, if $\psi$ is an
    $\eps$-representation of $\mcG'$, then 
    \begin{align*}
        \norm{\psi(J') - \Id}_f & \leq \norm{\psi(J') - \psi([s,J])}_f
            + \norm{\psi([s, J]) - \Id}_f \\
        & \leq \eps + \norm{\psi(s J s^{-1}) (\psi(J^{-1}) - \Id)}_f 
                + \norm{\psi(s J s^{-1}) - \Id}_f \\
            & = \eps + 2 \norm{\psi(J) - \Id}_f.
    \end{align*}
    Since any $\eps$-representation of $\mcG'$ restricts to an
    $\eps$-representation of $\mcG$, we conclude that 
    \begin{equation*}
        \hlp(J'; \delta, \eps) \geq \hlp(J; (\delta-\eps)/2, \eps).
    \end{equation*}
    We conclude in particular that for $\delta = 2$ and $\eps \leq 1$,
    \begin{equation*}
        \hlp(J'; 2, \eps) \geq C' \exp(C / \eps^{\alpha})
    \end{equation*}
    for constants $0 < \alpha < 1/2$ and $C, C' > 0$. 
        
    Now by \cite[Theorem 3.1]{Sl16}, there is an embedding of $\mcG'$ in the
    solution group $\Gamma$ of some linear system game $G$, and furthermore
    this embedding can be chosen to send $J'$ to the distinguished element
    $J_{\Gamma}$ of $\Gamma$.  Since $J'$ is non-trivial in $\mcG'$,
    $J_{\Gamma}$ will be non-trivial in $\Gamma$, and hence the game $G$
    associated to $\Gamma$ will have commuting-operator value $1$ by
    \cite[Theorem 4]{CLS16}. It follows from \cite[Corollary 5.2]{SV17} that
    there is a constant $C_0$ such that any $\eps$-perfect strategy for $G$ has
    local Hilbert space dimension $\geq \hlp(J_{\Gamma}, 2, C_0 \eps^{1/4})$. 
    Finally, since $\mcG'$ is a subgroup of $\Gamma$, there is a constant $C_1$
    such that $\hlp(J_{\Gamma}, 2, \eps) \geq \hlp(J', 2, C_1 \eps)$ (see, e.g.,
    \cite[Lemma 2.4]{SV17}), and the corollary follows.
\end{proof}
Another consequence of \cite{SV17} is that the game associated to a solution
group $\Gamma$ has quantum value $1$ if and only if the distinguished element
$J_{\Gamma}$ is non-trivial in approximate representations of $\Gamma$. 
Since we do not know if the group constructed $\mcG$ constructed in the proof
of Theorem \ref{T:main} is hyperlinear, we do not know if the element $J$ is
non-trivial in approximate representations of $\mcG$. Since the
class of hyperlinear groups is closed under HNN extensions by amenable subgroups,
if $J$ is non-trivial in approximate representations of $\mcG$, then the element
$J'$ in the proof of Corollary \ref{C:game} will be non-trivial in approximate
representations of $\mcG'$. However, this would still not be enough to show
that $J_{\Gamma}$ is non-trivial in approximate representations of $\Gamma$, 
since it is not known if the embedding constructed in \cite[Theorem 3.1]{Sl16}
preserves the property of being non-trivial in approximate representations.

\bibliographystyle{amsalpha}
\bibliography{approx}

\end{document}

%% file: header/ams_header.tex
\usepackage{amsmath,amsthm,amsfonts,amssymb,latexsym,verbatim,bbm,hyperref}
\usepackage[shortlabels]{enumitem}
\usepackage[top=1.5in,bottom=1.5in,left=1.5in,right=1.5in]{geometry}
\usepackage{subcaption}
\usepackage{afterpage}

\DeclareRobustCommand{\SkipTocEntry}[5]{}

\allowdisplaybreaks

\setlist{itemsep=.5\baselineskip,topsep=.5\baselineskip}

\numberwithin{equation}{section}
\theoremstyle{plain}
\newtheorem{theorem}{Theorem}[section]
\newtheorem{thm}[theorem]{Theorem}
\newtheorem{lemma}[theorem]{Lemma}

\newtheorem{defn}[theorem]{Definition}

\newtheorem{cor}[theorem]{Corollary}

\newtheorem{rmk}[theorem]{Remark}

\newcommand{\iso}{\cong}
\newcommand{\arr}{\rightarrow}

\newcommand{\R}{\mathbb{R}}
\newcommand{\C}{\mathbb{C}}
\newcommand{\Z}{\mathbb{Z}}

\newcommand{\eps}{\epsilon}

\newcommand{\Id}{\mathbbm{1}}

\newcommand{\mcC}{\mathcal{C}}

\newcommand{\mcF}{\mathcal{F}}
\newcommand{\mcG}{\mathcal{G}}

\newcommand{\mcK}{\mathcal{K}}

\newcommand{\mcU}{\mathcal{U}}

\newcommand{\mbN}{\mathbb{N}}


%% file: subexp.bbl
\providecommand{\bysame}{\leavevmode\hbox to3em{\hrulefill}\thinspace}
\providecommand{\MR}{\relax\ifhmode\unskip\space\fi MR }
\providecommand{\MRhref}[2]{%
  \href{http://www.ams.org/mathscinet-getitem?mr=#1}{#2}
}
\providecommand{\href}[2]{#2}
\begin{thebibliography}{DCGLT17}

\bibitem[BORS02]{BORS02}
J.-C. Birget, A.~Yu Ol'shanskii, E.~Rips, and M.~V. Sapir, \emph{Isoperimetric
  {Functions} of {Groups} and {Computational} {Complexity} of the {Word}
  {Problem}}, Annals of Mathematics \textbf{156} (2002), no.~2, 467--518.

\bibitem[CLS17]{CLS16}
Richard Cleve, Li~Liu, and William Slofstra, \emph{Perfect commuting-operator
  strategies for linear system games}, Journal of Mathematical Physics
  \textbf{58} (2017), 012202.

\bibitem[Cor13]{Co13}
Yves Cornulier, \emph{Sofic profile and computability of {Cremona} groups}, The
  Michigan Mathematical Journal \textbf{62} (2013), no.~4, 823--841.

\bibitem[DCGLT17]{DCGLT17}
Marcus De~Chiffre, Lev Glebsky, Alex Lubotzky, and Andreas Thom,
  \emph{Stability, cohomology vanishing, and non-approximable groups},
  arXiv:1711.10238.

\bibitem[DCOT17]{DOT17}
M.~De~Chiffre, N.~Ozawa, and A.~Thom, \emph{Operator algebraic approach to
  inverse and stability theorems for amenable groups}, arXiv:1706.04544.

\bibitem[Fri13]{Fr13}
Tobias Fritz, \emph{On infinite-dimensional state spaces}, Journal of
  Mathematical Physics \textbf{54} (2013), no.~5, 052107.

\bibitem[GH17]{GH17}
W.~T. Gowers and O.~Hatami, \emph{Inverse and stability theorems for
  approximate representations of finite groups}, Matematicheskii Sbornik
  \textbf{208} (2017), no.~12, 70--106.

\bibitem[Hig61]{Hi61}
G.~Higman, \emph{Subgroups of finitely presented groups}, Proceedings of the
  Royal Society of London. Series A, Mathematical and Physical Sciences
  \textbf{262} (1961), no.~1311, 455--475.

\bibitem[JLV18]{JLV18}
Zhengfeng Ji, Debbie Leung, and Thomas Vidick, \emph{A three-player coherent
  state embezzlement game}, arXiv: 1802.04926.

\bibitem[Kaz82]{Ka82}
D.~Kazhdan, \emph{On $\eps$-representations}, Israel Journal of Mathematics
  \textbf{43} (1982), no.~4, 315--323.

\bibitem[OV16]{ostrev2016entanglement}
Dimiter Ostrev and Thomas Vidick, \emph{Entanglement of approximate quantum
  strategies in {XOR} games}, arXiv preprint arXiv:1609.01652 (2016).

\bibitem[Slo16]{Sl16}
William Slofstra, \emph{Tsirelson's problem and an embedding theorem for groups
  arising from non-local games}, preprint (arXiv:1606.03140).

\bibitem[SV17]{SV17}
William Slofstra and Thomas Vidick, \emph{Entanglement in non-local games and
  the hyperlinear profile of groups}, arXiv:1711.10676.

\bibitem[Tho12]{Th12}
Andreas Thom, \emph{About the metric approximation of {Higman}'s group},
  Journal of Group Theory \textbf{15} (2012), no.~2, 301--310.

\end{thebibliography}
